\documentclass[a4paper,11pt]{article}
\usepackage{amsmath}
\usepackage{amsthm}
\usepackage{amsfonts}
\usepackage{amssymb}
\usepackage[latin1]{inputenc}
\usepackage[left=2.5cm,top=2.5cm,right=2.5cm,bottom=2.5cm]{geometry}

\newtheorem{lem}{Lemma}[section]
\newtheorem{thm}{Theorem}[section]
\newtheorem{co}[thm]{Corollary}
\newtheorem{pr}[thm]{Proposition} 
\numberwithin{equation}{section}

\begin{document}
\title{  Convergence in $L^p$ and its exponential rate  for a branching process in a random environment }
\author{Chunmao HUANG$^{a}$, Quansheng LIU$^{b,}$\footnote{Corresponding author at:  LMBA, UMR 6205, Universit\'e de Bretagne-Sud, Campus de Tohannic,
BP 573, 56017 Vannes, France.  \newline \indent \ \ Email addresses: cmhuang@hitwh.edu.cn (C. Huang), quansheng.liu@univ-ubs.fr (Q. Liu).}
\\
\small{\emph{$^{a}$Department of mathematics, Harbin institute of technology at Weihai, Weihai, 264209, China}}\\
\small{\emph{$^{b}$LMBA, Universit\'e de Bretagne-Sud, Campus de Tohannic,
BP 573, 56017 Vannes, France}}}
\maketitle

\begin{abstract}
We consider a supercritical branching process $(Z_n)$ in a random
environment $\xi$. Let $W$ be the limit of the normalized population size
$W_n=Z_n/\mathbb{E}[Z_n|\xi]$.  We first show a necessary  and sufficient condition for the quenched $L^p$ ($p>1$) convergence of $(W_n)$, which completes the known result for the annealed $L^p$ convergence. We then show that the convergence rate is exponential, and we find  the maximal value of $\rho>1$ such that $\rho^n(W-W_n)\rightarrow 0$  in $L^p$, in both quenched and annealed sense. Similar results are also shown for a branching process in a varying environment.
\\*



\end{abstract}

\section{Introduction and main results}
We consider a branching process in a random environment (BPRE). Let $\xi=(\xi_0, \xi_1, \xi_2,\cdots)$ be a stationary and ergodic process taking values in some space $\Theta$. Each realization of $\xi_n$ corresponds to a probability distribution on $\mathbb N_0=\{0,1, 2,\cdots\}$, denoted by $p(\xi_n)=\{p_i(\xi_n):i\in\mathbb N_0\}$, where
$$
 \text{$p_i(\xi_n)\geq0$, \quad$\sum_ip_i(\xi_n)=1$\quad and \quad$\sum_iip_i(\xi_n)\in(0,\infty)$. }
$$
The sequence $\xi=(\xi_n)$ will be called \emph{random environment}. A branching process $(Z_n)$ in the  random environment $\xi$ is a class of branching processes in varying environment indexed by $\xi$. By definition,
\begin{equation}
Z_0=1,\qquad Z_{n+1}=\sum_{i=1}^{Z_n}X_{n,i}\quad (n\geq0),
\end{equation}
where $X_{n,i} (i=1,2,\cdots)$ denotes the number of offspring of the
$i$th particle in the $n$th generation.  Given $\xi$, $\{X_{n,i}:n\geq0,i\geq1\}$ is a family of (conditionally)  independent
random variables and each $X_{n,i}$ has distribution $p(\xi_n)$ on $\mathbb N_0=\{0,1,\cdots\}$.

Let $(\Gamma, \mathbb{P}_\xi)$ be the probability space under which the
process is defined when the environment $\xi$ is fixed.  As usual,
$\mathbb{P}_\xi$ is called \emph{quenched law}.
The total probability space can be formulated as the product space
$( \Theta^{\mathbb{N}_0}\times\Gamma , \mathbb{P})$,
 where $ \mathbb{P} = \tau\otimes \mathbb{P}_{\xi} $ in the sense  that for all measurable and
 positive $g$, we have
  $$\int g d\mathbb{P} = \int_{\Theta^\mathbb{N }} \int_\Gamma g(\xi,y) d\mathbb{P}_{\xi}(y) d \tau  (\xi),$$
  where $\tau$ is the law of the environment $\xi$.  The total
probability $\mathbb{P}$ is usually called \emph{annealed law}.  The quenched law $\mathbb{P}_\xi$ may be considered to be the conditional
probability of $\mathbb{P}$ given $\xi$.

Let $\mathcal {F}_0=\sigma(\xi_0,\xi_1,\xi_2,\cdots)$ and $\mathcal
{F}_n=\sigma(\xi_0,\xi_1,\xi_2,\cdots,X_{k,i},\;0\leq k\leq
n-1,\;i=1,2,\cdots)$ be the $\sigma$-field generated by  $X_{k,i}$ ($0\leq k\leq n-1,\;i=1,2,\cdots$), so that $Z_n$ are $\mathcal {F}_n$-measurable.  For $n\geq0$ and $p\geq1$, set
\begin{equation}\label{CRlp1.1}
m_n(p)=\sum_ii^pp_i(\xi_n),\quad m_n=m_n(1),
\end{equation}
and
\begin{equation}\label{CRlp1.2}
P_0=1, \qquad P_n=\prod_{i=0}^{n-1}m_i \;\;(n\geq 1).
\end{equation}
So $m_n(p)=\mathbb{E}_\xi X_{n,i}^p$ and $P_n=\mathbb{E}_\xi Z_n$. It is well known that the normalized population size
\begin{equation}
W_n=\frac{Z_n}{P_n}
\end{equation}
 is a non-negative martingale with respect to $\mathcal F_n$ both under $\mathbb{P}_\xi$ for every $\xi$
and under $\mathbb{P}$,   hence the limit
\begin{equation}W=\lim_{n\rightarrow\infty}W_n\end{equation}
exists almost surely (a.s.) with $\mathbb{E}W\leq1$ by Fatou's lemma. Assume throughout the paper that  the process is supercritical in the sense that $\mathbb{E}\log m_0$ is well defined with
$$\mathbb{E}\log m_{0}>0. $$
We are interested in the  $L^p$ convergence rate of $W_n$ both in the quenched sense (under $\mathbb{P}_\xi$)  and  in the annealed sense (under $\mathbb{P}$).
\\*

 We first show a criterion for the quenched $L^p$ convergence of $W_n$.
\begin{thm}[Quenched $L^p$ convergence]\label{CRT2.1.3}
Let $ p>1$.  Consider the following assertions:
\begin{equation*}
\begin{array}{ll}
(i) \;\text{$\mathbb{E}\log \mathbb{E}_\xi\left(\frac{Z_1}{m_0}\right)^p<\infty$}; & (ii) \;\sup_n\mathbb{E}_\xi W_n^p<\infty \;a.s.; \\
(iii) \;\text{$W_n\rightarrow W$ in
$L^p$ under $\mathbb{P}_\xi$ for almost all $\xi$}; & (iv) \;0<\mathbb{E}_\xi W^p<\infty\; a.s..
\end{array}
\end{equation*}
Then the following implications hold: (i) $\Rightarrow$ (ii) $\Leftrightarrow$ (iii) $\Leftrightarrow$ (iv). If additionally $(\xi_n)$ are i.i.d. and $\mathbb{E}\log m_{0}<\infty$, then all the four assertions   are equivalent.
\end{thm}

It can be easily seen that  $\forall p>0$, $\mathbb{E}\log
\mathbb{E}_\xi(\frac{Z_1}{m_0})^p<\infty$ if and only if $\mathbb{E}\log^+\mathbb{E}_\xi|\frac{Z_1}{m_0}-1|^p<\infty$, where and hereafter we use the following usual notations:
$$\log^+x=\max(\log x, 0), \quad\log^-x=\max(-\log x, 0). $$

We will also use the notations
$$ a\wedge b=\min(a,b),\quad a\vee b=\max(a,b).$$

Next we  give a description of the  quenched $L^p$ convergence rate.
\begin{thm}[Exponential rate of quenched $L^p$ convergence]\label{CRC2.1.5}
Let $ p>1$,  $\rho>1$ and $m=\exp({\mathbb{E}\log m_0})>1$.
\begin{itemize}
\item[] (a) If $\mathbb{E}\log \mathbb{E}_\xi\left(\frac{Z_1}{m_0}\right)^p<\infty$, then
$$\lim_{n\rightarrow\infty}\rho^n(\mathbb{E}_\xi|W-W_n|^p)^{1/p}=0\quad a.s.\qquad\text{for}\;\rho<\min\{m^{1-1/p}, m^{1/2}\}.$$
\item[] (b) If $\mathbb{E}\log^-
\mathbb{E}_\xi\left|\frac{Z_1}{m_0}-1\right|^{p\wedge2}<\infty$ and  $\mathbb{E}\log^+
\mathbb{E}_\xi\left|\frac{Z_1}{m_0}-1\right|^{p\vee2}<\infty$, then a.s.
$$\limsup_{n\rightarrow\infty}\rho^n(\mathbb{E}_\xi|W-W_n|^p)^{1/p}\left\{\begin{array}{ll}
=0 \quad&\text{if}\;\rho<\bar{\rho}_c,\\
>0\quad &\text{if}\;\rho>\bar{\rho}_c,
\end{array}
\right.
$$
where $\bar{\rho}_c=m^{1/2}=\exp(\frac{1}{2}\mathbb{E}\log m_0)>1$.
\end{itemize}
\end{thm}

We mention that the theorem is valid with evident interpretation even if $\mathbb{E}\log m_{0}=\infty$ (so that $m=\infty$).

 Theorem \ref{CRC2.1.5}(a) shows that $W_n\rightarrow W$ in $L^p$ under $\mathbb{P}_\xi$ at an exponential rate;  Theorem \ref{CRC2.1.5}(b) means that $\bar{\rho}_c$ is the critical value of $\rho>1$ for which $\rho^n(W-W_n)\rightarrow0$ in $L^p$ under $\mathbb{P}_\xi$ for almost all $\xi$.

For the classical  Galton-Watson process,  Theorem \ref{CRC2.1.5}(a) reduces to the result of Liu (2001, \cite{liu3}) that if $\mathbb{E}Z_1^p<\infty$, then $\rho^n(W-W_n)\rightarrow0$ in
$L^p$ for $1<\rho<\min\{m^{1-1/p},m^{1/2}\}$, where $m=\mathbb{E}Z_1\in(1,\infty)$;  Theorem \ref{CRC2.1.5}(b) can be obtained by a result of Alsmeyer,   Iksanov,  Polotsky and R\"osler (2009, \cite{IK}) on branching random walks.
\\*

Recall that for a Galton-Watson process with $m=\mathbb{E}Z_1\in(1,\infty)$ and $\mathbb{P}(W>0)>0$, Asmussen (1976, \cite{soren}) showed that for $p\in(1,2)$, $W-W_n=o(m^{-n/q})$ a.s. if and  only if $\mathbb{E}Z_1^p<\infty$, where $1/p+1/q=1$. As an application of  Theorem \ref{CRC2.1.5},  we immediately obtain the following similar result for a branching process  in a random environment.
\begin{co}[Exponential rate of a.s. convergence]\label{CRCC1}
Let $p\in(1,2)$ and $m=\exp(\mathbb{E}\log m_0)\in(0,\infty)$. If $\mathbb{E}\log \mathbb{E}_\xi(\frac{Z_1}{m_0})^p<\infty$, then $\forall \varepsilon>0$,
\begin{equation}\label{CREE1}
W-W_n=o(m^{-\frac{n}{q+\varepsilon}})\quad a.s.,
\end{equation}
where $1/p+1/q=1$.
\end{co}

In fact, to see the conclusion,   let $\rho_1=m^{\frac{1}{q+\varepsilon}}$ and take $\rho$ satisfying $\rho_1<\rho<m^{1/q}$. By Theorem \ref{CRC2.1.5}(a), $\rho^n(\mathbb{E}_\xi|W-W_n|^p)^{1/p}\rightarrow0$, so that
$$\mathbb{E}_\xi\left(\sum_n\rho_1^n|W-W_n|\right)\leq\sum_n\left(\frac{\rho_1}{\rho}\right)^n\rho^n(\mathbb{E}_\xi|W-W_n|^p)^{1/p}<\infty\;\;a.s..$$
Therefore the series $\sum_n\rho_1^n(W-W_n)$ converges a.s., which implies (\ref{CREE1}).
 \\*

Corollary \ref{CRCC1} has recently been shown by Huang and Liu (\cite{huang2}, 2013) by a truncating argument. The approach here is quite different.
\\*

We now turn to the annealed $L^p$ convergence of $W_n$. When the environment is $i.i.d.$, a necessary and sufficient condition was shown
by Guivarc'h and Liu (2001, \cite{gliu}, Theorem 3).

\begin{pr}[Annealed $L^p$ convergence \cite{gliu}]
Assume that $(\xi_n)$ are i.i.d. and $ p>1$. Then the following assertions are equivalent:
\begin{equation*}
\begin{array}{ll}
(i) \;\text{$\mathbb{E}\left(\frac{Z_1}{m_0}\right)^p<\infty \;\text{and}\; \mathbb{E}m_0^{1-p}<1$}; & (ii) \;\sup_n\mathbb{E}W_n^p<\infty ; \\
(iii) \;\text{$W_n\rightarrow W$ in
$L^p$ under $\mathbb{P}$}; & (iv) \;0<\mathbb{E} W^p<\infty.
\end{array}
\end{equation*}
\end{pr}

We shall prove the following theorem about the rate of convergence.

\begin{thm}[Exponential rate of annealed $L^p$  convergence]\label{CRC5.9}
 Assume that $(\xi_n)$ are i.i.d.. Let $p>1$ and
$\rho>1$.
\begin{itemize}
\item[] (a) Assume that $\mathbb{E}\left(\frac{Z_1}{m_0}\right)^p<\infty$ and $\mathbb{E}m_0^{1-p}<1$. Then
$$\lim_{n\rightarrow\infty}\rho^n(\mathbb{E}|W-W_n|^p)^{1/p}=0\quad\text{for}\;\rho<\rho_0,$$
where $\rho_0>1$ is defined by
\begin{equation*}
\rho_0=\left\{\begin{array}{ll}
(\mathbb{E}m_0^{1-p})^{-1/p}& \text{if $p\in(1, 2)$},\\
\min\{(\mathbb{E}m_0^{1-p})^{-1/p},(\mathbb{E}m_0^{-p/2})^{-1/p}\}&\text{if $p\geq 2$}.
\end{array}\right.
\end{equation*}
 \item[](b) Assume that $\mathbb{P}(W_1=1)<1$ and that either of the following conditions is satisfied:
\begin{equation*}
\begin{array}{l}
\text{(i) $p\in(1,2)$, $\mathbb{E}\left(\mathbb{E}_\xi\left(\frac{Z_1}{m_0}\right)^2\right)^{p/2}<\infty$, $\mathbb{E}m_0^{-p/2}\log m_0>0$ and $\mathbb{E}m_0^{-p/2-1}Z_1\log^+Z_1<\infty$};\\
\text{(ii) $p\geq 2$ and $\mathbb{E}(\frac{Z_1}{m_0})^p<\infty$}.
\end{array}
\end{equation*}
Set
\begin{equation*}
\rho_c=\left\{\begin{array}{ll}
(\mathbb{E}m_0^{-p/2})^{-1/p}& \text{if $p\in(1, 2)$},\\
\min\{(\mathbb{E}m_0^{1-p})^{-1/p},(\mathbb{E}m_0^{-p/2})^{-1/p}\}&\text{if $p\geq 2$}.
\end{array}\right.
\end{equation*}
Then
$$\limsup_{n\rightarrow\infty}\rho^n(\mathbb{E}|W-W_n|^p)^{1/p}\left\{\begin{array}{ll}
=0\quad  &\text{if}\;\rho<\rho_c,\\
>0\quad   &\text{if}\;\rho>\rho_c.
\end{array}
\right.
$$
\end{itemize}
\end{thm}

\noindent\textbf{Remark.} By the convexity of the function $\mathbb{E}m_0^{-x}$, the condition $\mathbb{E}m_0^{-p/2}\log m_0>0$ implies that $\mathbb{E}m_0^{-x}$ is strictly decreasing on $(-\infty, \frac{p}{2}]$. Thus $\mathbb{E}m_0^{-p/2}<\mathbb{E}m_0^{1-p}$ for $p\in(1,2)$, so that $\rho_0\leq\rho_c$.
\\*


Theorem \ref{CRC5.9}(a) implies that $W_n\rightarrow W$   in $L^p$ under $\mathbb{P}$ (annealed) at an exponential rate.  Theorem \ref{CRC5.9}(b) shows that under certain moment conditions,   $\rho_c$  is the critical value of $\rho>1$ for the annealed $L^p$ convergence of  $\rho^n(W-W_n)$ to $0$, while
Theorem \ref{CRC2.1.5}(b) shows that $\bar\rho_c$ is the critical value for the quenched $L^p$ convergence.
Notice that by Jensen's inequality,
$$\mathbb{E}m_0^{-p/2}=\mathbb{E}\exp(-\frac{p}{2}\log m_0)\geq \exp(-\frac{p}{2}\mathbb{E}\log m_0),$$
so that $(\mathbb{E}m_0^{-p/2})^{-1/p}\leq \exp(\frac{1}{2}\mathbb{E}\log m_0)$. This shows that $\rho_c\leq \bar{\rho}_c$.
\\*

The rest of this paper is organized as follows. In Section \ref{CRLPS2}, we consider the $L^p$ convergence of the martingale $W_n$ and its exponential rate for a branching process in a varying environment.   In Sections \ref{CRLPS5}
and \ref{CRLPS6}, we study the random environment case, and give the proofs of the main results:
in Section \ref{CRLPS5}, we consider the quenched case and prove Theorems \ref{CRT2.1.3} and \ref{CRC2.1.5}; in Section \ref{CRLPS6}, we consider the annealed case and give the proof of Theorem \ref{CRC5.9}.

\section{Branching process in a varying environment}\label{CRLPS2}
In this section,  as preliminaries, we study the $L^p$ convergence and its convergence rate for a branching process $(Z_n)$ in a varying environment (BPVE).
 By definition,
\begin{equation}\label{CRLPEE21}
Z_0=1,\qquad Z_{n+1}=\sum_{i=1}^{Z_n}X_{n,i}\quad (n\geq0),
\end{equation}
where $X_{n,i} (i=1,2,\cdots)$ denotes the number of offspring of the
$i$th particle in the $n$th generation, each $X_{n,i}$ has distribution
$p(n)=\left\{p_i(n):i\in\mathbb N_0\right\}$
 on $\mathbb N_0=\{0,1,\cdots\}$, where
 $$
 \text{$p_i(n)\geq0$, \quad$\sum_ip_i(n)=1$ \quad and \quad$\sum_iip_i(n)\in(0,\infty)$;}
 $$
 all the random variables $X_{n,i} (n\geq 0, i\geq1)$ are independent of each other. Let $(\Gamma,\mathbb{P})$ be  the underlying probability space.
  For $n\geq0$ and $p\geq1$, set
\begin{equation}
m_n(p)=\mathbb{E}X_{n,i}^p=\sum_{k} k^pp_k(n),\qquad m_n=m_n(1),
\end{equation}
 and
\begin{equation}
\bar{m}_n(p)=\mathbb{E}\left|\frac{X_{n,i}}{m_n}-1\right|^p=\sum_{k} \left|\frac{k-m_n}{m_n}\right|^pp_k(n).
\end{equation}

Let $\mathcal {F}_0=\{\emptyset,\Gamma\}$ and $\mathcal
{F}_n=\sigma(X_{k,i}:\;0\leq k\leq n-1,\;i=1,2,\cdots)$ be the
$\sigma$-field generated by $X_{k,i} \;(0\leq
k\leq n-1,\;i=1,2,\cdots)$. Like the case of BPRE, let
\begin{equation}
P_0=1,\qquad P_n=\prod_{i=0}^{n-1}m_i\;\;(n\geq 1).
\end{equation}
Then the normalized population size
$W_n={Z_n}/{P_n}$
is a non-negative martingale with
respect to the filtration $\mathcal{F}_n$, and
$\lim_{n\rightarrow\infty}W_n=W$ a.s. for some non-negative random
variable $W$ with $\mathbb{E}W\leq1$. It is well known that there is a
non-negative but possibly infinite random variable $Z_\infty$ such
that $Z_n\rightarrow Z_\infty$ in distribution as
$n\rightarrow\infty$. We are interested in the supercritical case
where $\mathbb{P}(Z_\infty=0)<1$, so that by (\cite{jajer}, Corollary 3), either
$\sum\limits_{n=0}^\infty(1-p_1(n))<\infty$, or
$\lim\limits_{n\rightarrow\infty}P_n=\infty$. Here we assume that
$\lim\limits_{n\rightarrow\infty}P_n=\infty$.

We are interested in the $L^p$ convergence of the martingale
$W_n$ and its convergence rate.  We
have the following theorems.

\begin{thm}[$L^p$ convergence of $W_n$ for BPVE]\label{CRT1.3}
Let $(Z_n)$ be the BPVE defined in (\ref{CRLPEE21}).
\begin{itemize}
\item[](i) Let $p\in(1,2)$. If
$\sum_nP_n^{p(1/r-1)} \bar{m}_n(r)^{p/r}<\infty$
for some $r\in[p,2]$, then
\begin{equation}\label{CREE2.1}
\text{$W_n\rightarrow W$ \quad in \;\;$L^p$.}
\end{equation}
Conversely, if $\liminf\limits_{n\rightarrow\infty}\frac{\log
P_n}{n} >0$ and (\ref{CREE2.1}) holds, then
$\sum_nP_n^{-s-p/2}\bar{m}_n(p)<\infty$ for all $ s>0$.

\item[](ii) Let $p\geq2$. If
$\sum_nP_n^{-1}\bar{m}_n(p)^{2/p}<\infty$,
then (\ref{CREE2.1}) holds.
Conversely, if (\ref{CREE2.1}) holds, then
$\sum_nP_n^{p(1/r-1)}\bar{m}_n(r)^{p/r}<\infty$ for all $r\in[2,p]$.
\end{itemize}
\end{thm}

\noindent\textbf{Remark 1.}  As $W_n$ is a martingale, $\forall p>1$, (\ref{CREE2.1}) holds if and only if $\sup_n\mathbb{E}W_n^p<\infty$.
\\*

\noindent\textbf{Remark 2.} For $p=2$, $\sup_n\mathbb{E}W_n^2=1+\sum_{n=0}^{\infty}P_n^{-1} \bar{m}_n(2)$.
So $\sup_n\mathbb{E}W_n^2<\infty$ if and only if $\sum_nP_n^{-1}\bar{m}_n(2)<\infty$, as shown by Jagers (1974,  \cite{jajer}, Theorem 4).
\\*



\begin{thm}[Exponential   rate of $L^p$ convergence of $W_n$ for BPVE]\label{CRC1.5}
Let $(Z_n)$ be the BPVE defined in (\ref{CRLPEE21}) and let $\rho>1$.
\begin{itemize}
\item[] (i) Let $p\in(1,2)$. If
$\sum_n\rho^{pn}P_n^{p(1/r-1)}\bar{m}_n(r)^{p/r}<\infty$ for some
$r\in[p, 2]$, then
\begin{equation}\label{CRE1.1.5}
(\mathbb{E}|W-W_n|^p)^{1/p}=o(\rho^{- n}).
\end{equation}
 Conversely, if $\liminf_{n\rightarrow\infty}\frac{\log P_n}{n} >0$ and
 (\ref{CRE1.1.5}) holds, then $\sum_n{\rho_1}^{p
n}P_n^{-s-p/2}\bar{m}_n(p)<\infty$ for all  $\rho_1\in(1,\rho)$ and all
$s>0$.
\item[] (ii) Let $p\geq2$. If $\sum_n\rho^{2
n}P_n^{-1}\bar{m}_n(p)^{2/p}<\infty$, then (\ref{CRE1.1.5})  holds. Conversely, if (\ref{CRE1.1.5}) holds, then $\sum_n{\rho_1}^{pn}P_n^{p(1/r-1)}\bar{m}_n(r)^{p/r}<\infty$ for all $\rho_1\in(1,\rho)$ and $ r\in[2,p]$.
\end{itemize}
\end{thm}

\subsection{The martingale $\{\hat{A}_n\}$}\label{CRLPS3}
To estimate the exponential rate of  $L^p$ convergence of $W_n$, following \cite{IK}, we consider the series
\begin{equation}\label{CRE1.6}
A(\rho)=\sum_{n=0}^{\infty}\rho^{n}(W-W_n)\quad (\rho>1).
\end{equation}
Here $A(\rho)$  denotes the series; it will also denote the sum of the series when the series converges.
The convergence of the series $A(\rho)$ reflects the exponential
rate of $W-W_n$. More precisely, if the series $A(\rho)$ converges a.s. (resp. in $L^p$, $p>1$), then
$\rho^{n}(W-W_n)\rightarrow0$ a.s. (resp. in $L^p$). Conversely,
if $\rho^{n}(W-W_n)\rightarrow0$ a.s. (resp. in $L^p$), then for $\rho_1\in(1,\rho)$, the series $A(\rho_1)$ converges a.s. (resp. in $L^p$). Moreover, from the remark after Lemma \ref{CRL1.3.2} below, we shall see that the $L^p$ convergence of the series $A(\rho)$ implies its a.s. convergence.

As in \cite{IK}, we  introduce an associated martingale $\{\hat{A}_n\}$. Let $\rho\geq1$, and define
\begin{equation}
\hat{A}_n=\hat{A}_n(\rho)=\sum_{k=0}^{n}\rho^{k}(W_{k+1}-W_k),\qquad
\hat{A}(\rho)=\sum_{n=0}^{\infty}\rho^{n}(W_{n+1}-W_n).
\end{equation}
As in the case of $A(\rho)$, here $\hat A(\rho)$ also denotes the series and it also denote the sum of the series when the series converges.
It is easy to see that $\{(\hat{A}_n;\mathcal{F}_{n+1})\}$ forms a
martingale. In particular, for $\rho=1$, $\hat{A}_n=W_{n+1}-1$.
By the convergence theorems for martingales,
$\sup_n\mathbb{E}|\hat{A}_n|^p<\infty$ implies that the series $\hat{A}(\rho)$ converges a.s.
and in $L^p$. Therefore the $L^p$ convergence of $\hat{A}(\rho)$ is
equivalent to $\sup_n\mathbb{E}|\hat{A}_n|^p<\infty$. Moreover, if $\hat{A}(\rho)$ converges
in $L^p$, then it also converges a.s..

It is known that the series $A(\rho)$ and $\hat A(\rho)$ have the following relations.

\begin{lem}[\cite{IK}, Lemma 3.1]\label{CRL1.3.2}
Let $p>1$ and $\rho>1$.
 The series $A(\rho)$ converges a.s. (resp. in $L^p$) if and only if the same is true for the series $\hat{A}(\rho)$.
\end{lem}

\noindent\textbf{Remark.} According to the relations between $\hat{A}_n$ and $\hat A(\rho)$ stated above,
Lemma
\ref{CRL1.3.2} in fact tells us that $A(\rho)$ converges in $L^p$ if and
only if $\sup_n\mathbb{E}|\hat{A}_n|^p<\infty$, and the $L^p$ convergence of
$A(\rho)$ implies its a.s. convergence.
\\*

We shall study the $L^p$ convergence of $A(\rho)$ through the existence of the $p$th-moment of the martingale $\{\hat A_n\}$. The main tool is the Burkholder's inequality for martingales.

\begin{lem}[Burkholder's inequality, see e.g. \cite{chow}]\label{CRL1.3.1}
Let $\{S_n\}$ be a $L^1$ martingale with $S_0=0$. Let $Q_n=(\sum_{k=1}^{n}(S_k-S_{k-1})^2)^{1/2}$ and
$Q=(\sum_{n=1}^{\infty}(S_n-S_{n-1})^2)^{1/2}$. Then $\forall p>1$,
$$a_p\parallel Q_n\parallel_p\leq\parallel S_n\parallel_p\leq b_p\parallel Q_n\parallel_p,$$
$$a_p\parallel Q\parallel_p\leq\sup_n\parallel S_n\parallel_p\leq b_p\parallel Q\parallel_p,$$
where $a_p=(p-1)/18p^{3/2}$, $b_p=18p^{3/2}/(p-1)^{1/2}$.
\end{lem}

The following lemma gives the relations  between $\sup_n\mathbb{E}|\hat{A}_n|^p$ and
$\mathbb{E}|W_{n+1}-W_n|^p$ that we shall use later.

\begin{lem}\label{CRL1.3.3}
Let $p>1$and write $a_p=(p-1)/18p^{3/2}$, $b_p=18p^{3/2}/(p-1)^{1/2}$. Then:
\begin{itemize}
\item[] (i) For $p\in(1,2)$ and $N\geq1$,
\begin{equation}\label{CRE1.3.2}
a_pN^{p/2-1}\sum_{n=0}^{N-1}\rho^{pn}\mathbb{E}|W_{n+1}-W_n|^p\leq\sup_n\mathbb{E}|\hat{A}_n|^p\leq
b_p\sum_{n=0}^{\infty}\rho^{pn}\mathbb{E}|W_{n+1}-W_n|^p.
\end{equation}

\item[] (ii) For $p=2$,
\begin{equation}\label{CRE1.3.3}
\sup_n\mathbb{E}|\hat{A}_n|^2=\sum_{n=0}^{\infty}\rho^{2n}\mathbb{E}|W_{n+1}-W_n|^2.
\end{equation}

\item[] (iii) For $p>2$,
\begin{equation}\label{CRE1.3.4}
a_p\sum_{n=0}^{\infty}\rho^{pn}\mathbb{E}|W_{n+1}-W_n|^p\leq
\sup_n\mathbb{E}|\hat{A}_n|^p \leq
b_p\left(\sum_{n=0}^{\infty}\rho^{2n}(\mathbb{E}|W_{n+1}-W_n|^p)^{2/p}\right)^{p/2}.
\end{equation}
\end{itemize}
\end{lem}

\begin{proof}Firstly,  (\ref{CRE1.3.3}) is obvious by the orthogonality of martingale. The upper bound in (\ref{CRE1.3.2}) and (\ref{CRE1.3.4}) are directly from Burkholder's inequality. The lower bound in  (\ref{CRE1.3.2}) can be obtained following similar arguments in \cite{IK} (p.25).
\end{proof}

\noindent\textbf{Remark.} For a BPRE, notice that $\{W_n\}$ is a martingale wit  both under $\mathbb{P}_\xi$ (for every $\xi$) and under $\mathbb{P}$, and the same is true for $\{\hat{A}_n\}$. Thus Lemmas \ref{CRL1.3.2} and
\ref{CRL1.3.3} hold for both expectations $\mathbb{E}_\xi$ and $\mathbb{E}$.
\\*

Let
\begin{equation}
\bar{X}_{n,i}=\frac{X_{n,i}}{m_n},\qquad\bar{X}_n=\bar{X}_{n,1}.
\end{equation}
From the definitions of $Z_n$ and $W_n$, we have
\begin{equation}\label{CRE1.1.1}
W_{n+1}-W_n=\frac{1}{P_n}\sum_{i=1}^{Z_n}(\bar{X}_{n,i}-1).
\end{equation}
This fact leads us to estimate $\mathbb{E}|W_{n+1}-W_n|^p$ through the moments of  $\bar X_n -1$.

\begin{lem}\label{CRL1.4.1}
Let $p>1$, $n\geq0$ and write $a_p=(p-1)/18p^{3/2}$, $b_p=18p^{3/2}/(p-1)^{1/2}$. Then:
\begin{itemize}
\item[](i) For $p\in(1,2)$ and $r\in[p,2]$,
\begin{equation}\label{CRE1.4.1}
a_pP_n^{-p/2}\mathbb{E}W_n^{p/2}\mathbb{E}|\bar{X}_n-1|^p\leq \mathbb{E}|W_{n+1}-W_n|^p\leq
b_pP_n^{p(1/r-1)}(\mathbb{E}|\bar{X}_n-1|^r)^{p/r}.
\end{equation}

\item[] (ii) For $p=2$,
\begin{equation}\label{CRE1.4.2}
\mathbb{E}|W_{n+1}-W_n|^2=P_n^{-1}\mathbb{E}|\bar{X}_n-1|^2.
\end{equation}

\item[](iii) For $p>2$ and $r\in[2,p]$,
\begin{equation}\label{CRE1.4.3}
a_pP_n^{p(1/r-1)}(\mathbb{E}|\bar{X}_n-1|^r)^{p/r}\leq \mathbb{E}|W_{n+1}-W_n|^p\leq
b_pP_n^{-p/2}\mathbb{E}W_n^{p/2}\mathbb{E}|\bar{X}_n-1|^p.
\end{equation}
\end{itemize}
\end{lem}

\begin{proof}
We first prove (ii). By (\ref{CRE1.1.1}),
\begin{eqnarray*}
\mathbb{E}|W_{n+1}-W_n|^2
&=&\frac{1}{P_n^2}\mathbb{E}\sum_{i=1}^{Z_n}(\bar{X}_{n,i}-1)\sum_{j=1}^{Z_n}(\bar{X}_{n,j}-1)\\
&=&\frac{1}{P_n^2}\mathbb{E}\sum_{i=1}^{Z_n}(\bar{X}_{n,i}-1)^2=P_n^{-1}\mathbb{E}|\bar{X}_n-1|^2.
\end{eqnarray*}
We then prove (i) and (iii).  Let $p>1$. Fix $n\geq0$ and let
 $$
 \text{$S_0=0$, \qquad $S_k=P_n^{-1}\sum_{i=1}^{k}(\bar{X}_{n,i}-1)\mathbf{1}_{\{Z_n\geq i\}}$.}
 $$
Let $\mathcal {G}_0=\mathcal {F}_n$ and $\mathcal {G}_k=\sigma(\mathcal {F}_n,X_{n,i},1\leq i\leq k)$.
It is not difficult to verify that $\{S_k\}$ forms a martingale with respect to $\mathcal {G}_k$
and $\{S_k\}$ is uniformly integrable, so that $\sup_k\mathbb{E}|S_k|^p=\mathbb{E}|S|^p$,
 where $S=\lim\limits_{k\rightarrow\infty}S_k=P_n^{-1}\sum_{i=1}^{Z_n}(\bar{X}_{n,i}-1)=W_{n+1}-W_n$. By Burkholder's inequality,
 $$a_p\mathbb{E}\left|\sum_{k=1}^{\infty} (S_{k}-S_{k-1})^2\right|^{p/2}\leq \mathbb{E}|S|^p
\leq b_p\mathbb{E}\left|\sum_{k=1}^{\infty} (S_{k}-S_{k-1})^2\right|^{p/2},$$ which
means that
\begin{equation}\label{CRE1.4.4}
a_p\mathbb{E}\left|\frac{1}{P_n^2}\sum_{i=1}^{Z_n}(\bar{X}_{n,i}-1)^2\right|^{p/2}\leq
\mathbb{E}|W_{n+1}-W_n|^p \leq
b_p\mathbb{E}\left|\frac{1}{P_n^2}\sum_{i=1}^{Z_n}(\bar{X}_{n,i}-1)^2\right|^{p/2}.
\end{equation}

For $p\in(1,2)$ and $r\in[p,2]$, by the concavity of $x^{r/2}$, $x^{p/r}$ and $x^{p/2}$, we have
\begin{eqnarray}\label{CRE1.4.5}
\mathbb{E}\left|\frac{1}{P_n^2}\sum_{i=1}^{Z_n}(\bar{X}_{n,i}-1)^2\right|^{p/2}
&=&\mathbb{E}\left|\frac{1}{P_n^2}\sum_{i=1}^{Z_n}(\bar{X}_{n,i}-1)^2\right|^{\frac{r}{2}\cdot\frac{p}{r}}\nonumber\\
&\leq&\mathbb{E}\left(P_n^{-r}\sum_{i=1}^{Z_n}|\bar{X}_{n,i}-1|^r\right)^{p/r}\nonumber\\
&\leq&P_n^{p(1/r-1)}(\mathbb{E}|\bar{X}_n-1|^r)^{p/r},
\end{eqnarray}
and
\begin{eqnarray}\label{CRE1.4.6}
\mathbb{E}\left|\frac{1}{P_n^2}\sum_{i=1}^{Z_n}(\bar{X}_{n,i}-1)^2\right|^{p/2}
&\geq&P_n^{-p}\mathbb{E}Z_n^{p/2-1}\sum_{i=1}^{Z_n}|\bar{X}_{n,i}-1|^p\nonumber\\
&=&P_n^{-p/2}\mathbb{E}W_n^{p/2}\mathbb{E}|\bar{X}_{n}-1|^p.
\end{eqnarray}
Combing (\ref{CRE1.4.5}), (\ref{CRE1.4.6}) with (\ref{CRE1.4.4}), we obtain (\ref{CRE1.4.1}).

For $p>2$ and $r\in[2,p]$, since $x^{r/2}$, $x^{p/r}$ and $x^{p/2}$ are convex,   (\ref{CRE1.4.5}) holds with $"\leq"$ replaced by $"\geq"$, while (\ref{CRE1.4.6}) holds with $"\geq"$ replaced by $"\leq"$.
\end{proof}


\subsection{Moments of $\{\hat A_n\}$; Proofs of Theorems \ref{CRT1.3} and \ref{CRC1.5}}\label{CRLPS4}
In this section, we study the $p$th-moment of $\{\hat A_n\}$, and prove Theorems \ref{CRT1.3} and \ref{CRC1.5}.

\begin{pr}[Moments of $\hat A_n$ for BPVE]\label{CRP1.4.1} Let $\rho\geq1$.
\begin{itemize}
\item[] (i) Let $p\in(1,2)$. If
$\sum_n\rho^{pn}P_n^{p(1/r-1)}\bar{m}_n(r)^{p/r}<\infty$ for some
$r\in[p,2]$, then
\begin{equation}\label{CREE4.1}
\sup_n\mathbb{E}|\hat{A}_n|^p<\infty.
\end{equation}
Conversely, if
$\liminf\limits_{n\rightarrow\infty}\frac{\log P_n}{n}>0$ and (\ref{CREE4.1}) holds, then $\sum_n\rho^{pn}P_n^{-s
-p/2}\bar{m}_n(p)<\infty$ for any $s>0$.
\item[](ii) Let $p\geq2$. If $\sum_n\rho^{2n}P_n^{-1}\bar{m}_n(p)^{2/p}<\infty$,
then (\ref{CREE4.1}) holds. Conversely,
if (\ref{CREE4.1}) holds, then for any $r\in[2,p]$,
$\sum_n\rho^{pn}P_n^{p(1/r-1)}\bar{m}_n(r)^{p/r}<\infty$.
\end{itemize}
\end{pr}

\noindent\textbf{Remark}. For $p=2$,
$\sup_n\mathbb{E}|\hat{A}_n|^2=\sum_{n=0}^{\infty}\rho^{2n}P_n^{-1}\bar{m}_n(2)$.
\\*

Before the proof of Proposition \ref{CRP1.4.1}, we give another lower bound  of $\sup_n\mathbb{E}|\hat{A_n}|^p$ for $p\in(1,2)$, which is different from (\ref{CRE1.3.2}).

\begin{lem}\label{CRL1.4.2}
Let $p\in(1,2)$ and $s>0$. If $\eta=\eta(s):=\sum_nP_n^{-s}<\infty$, then writing $a_p=(p-1)/18p^{3/2}$ and $b_p=18p^{3/2}/(p-1)^{1/2}$, we have
\begin{equation}\label{CRE1.4.8}
\sup_n\mathbb{E}|\hat{A_n}|^p\geq
a_p\eta^{p/2-1}\sum_{n=0}^{\infty}\rho^{pn}P_n^{s(p/2-1)}\mathbb{E}|W_{n+1}-W_n|^p.
\end{equation}
\end{lem}

\begin{proof}
Applying Burkholder's inequality and Jensen's inequality, we get
\begin{eqnarray*}
\sup_n\mathbb{E}|\hat A_n|^p&\geq &a_p\mathbb{E}\left|\sum_{n=0}^{\infty}\rho^{2n}(W_{n+1}-W_n)^2\right|^{p/2}\\
&=&a_p\mathbb{E}\left(\sum_{n=0}^{\infty}\frac{1}{\eta P_n^s}(\eta P_n^s\rho^{2n}|W_{n+1}-W_n|^2)\right)^{p/2}\\
&\geq&a_p\mathbb{E}\sum_{n=0}^{\infty}\frac{1}{\eta P_n^s}(\eta P_n^s\rho^{2n}|W_{n+1}-W_n|^2)^{p/2}\\
&=&a_p\eta^{p/2-1}\sum_{n=0}^{\infty}\rho^{pn}P_n^{s(p/2-1)}\mathbb{E}|W_{n+1}-W_n|^p.
\end{eqnarray*}
So (\ref{CRE1.4.8}) is proved.
\end{proof}

\begin{proof}[Proof of Proposition \ref{CRP1.4.1}]
(i) By Lemmas \ref{CRL1.3.3} and  \ref{CRL1.4.1}, for $r\in[p,2]$,
\begin{eqnarray*}
\sup_n\mathbb{E}|\hat{A}_n|^p&\leq&C\sum_{n=0}^{\infty}\rho^{pn}\mathbb{E}|W_{n+1}-W_n|^p\\
&\leq&C\sum_{n=0}^{\infty}\rho^{pn}P_n^{p(1/r-1)}(\mathbb{E}|\bar{X}_n-1|^r)^{p/r}.
\end{eqnarray*}
Here and throughout this paper $C$ denotes a general positive constant (maybe different from line to line).
Hence $\sup_n\mathbb{E}|\hat{A}_n|^p<\infty$,  if
$\sum_n\rho^{pn}P_n^{p(1/r-1)}\bar{m}_n(r)^{p/r}<\infty$ for some
$r\in[p,2]$. Conversely, assume that
$\sup_n\mathbb{E}|\hat{A}_n|^p<\infty$. For any $s>0$, let
$s'=\frac{2s}{2-p}>0$. Since $\eta=\eta(s')<\infty$, by
(\ref{CRE1.4.8}) and Lemma \ref{CRL1.4.1},
$$\sup_n\mathbb{E}|\hat{A}_n|^p\geq C\eta^{p/2-1}\inf_n\mathbb{E}W_n^{p/2}\sum_{n=0}^{\infty}\rho^{pn}P_n^{-s-p/2} \bar{m}_n(p).$$
Thus $\sum_n\rho^{pn}P_n^{-s-p/2} \bar{m}_n(p)<\infty$, $\forall s>0$.

(ii)  For $p=2$, by (\ref{CRE1.3.3}) and (\ref{CRE1.4.2}),
$$\sup_n\mathbb{E}|\hat{A}_n|^2=\sum_{n=0}^{\infty}\rho^{2n}P_n^{-1}\mathbb{E}|\bar{X}_n-1|^2=\sum_{n=0}^{\infty}\rho^{2n}P_n^{-1}\bar{m}_n(2).$$
Thus $\sup_n\mathbb{E}|\hat{A}_n|^2<\infty$ if and only if
$\sum_n\rho^{2n}P_n^{-1}\bar{m}_n(2)<\infty$.

Let $p>2$. We first
assume that
$\sum_n\rho^{2n}P_n^{-1}(\mathbb{E}|\bar{X}_n-1|^p)^{2/p}$ ($=\sum_n\rho^{2n}P_n^{-1}
\bar{m}_n(p)^{2/p}$ ) $<\infty$. By (\ref{CRE1.3.4} ) and
(\ref{CRE1.4.3}),
\begin{eqnarray}\label{CRLPEE41}
\sup_n\mathbb{E}|\hat{A}_n|^p&\leq&C\left(\sum_{n=0}^{\infty}\rho^{2n}(\mathbb{E}|W_{n+1}-W_n|^p)^{2/p}\right)^{p/2}\nonumber\\
&\leq&C\left(\sum_{n=0}^{\infty}\rho^{2n}P_n^{-1}(\mathbb{E}W_n^{p/2})^{2/p}(\mathbb{E}|\bar{X}_n-1|^p)^{2/p}\right)^{p/2}\nonumber\\
&\leq&C\sup_n\mathbb{E}W_n^{p/2}\left(\sum_{n=0}^{\infty}\rho^{2n}P_n^{-1}(\mathbb{E}|\bar{X}_n-1|^p)^{2/p}\right)^{p/2}<\infty,
\end{eqnarray}
provided that $\sup_n\mathbb{E}W_n^{p/2}<\infty$, which holds obviously when $\sup_n\mathbb{E}W_n^{p}<\infty$. Therefore, it suffices to prove that  for every integer $b\geq1$,
\begin{equation}\label{CRE1.4.9}
\sup_n\mathbb{E}W_n^p<\infty \quad \text{if}
\quad\sum_nP_n^{-1}(\mathbb{E}|\bar{X}_n-1|^p)^{2/p}<\infty, \quad
\forall p\in(2^b,2^{b+1}].
\end{equation}
We shall prove (\ref{CRE1.4.9}) by induction on b. For $b=1$, we consider $p\in(2, 2^2]$, so that $p/2\in(1,2]$. By H\"older's inequality,
$$\sum_nP_n^{-1}\mathbb{E}|\bar{X}_n-1|^2\leq\sum_nP_n^{-1}(\mathbb{E}|\bar{X}_n-1|^p)^{2/p}<\infty.$$
Hence $\sup_n\mathbb{E}W_n^2<\infty$, so that  $\sup_n\mathbb{E}W_n^{p/2}<\infty$. By (\ref{CRLPEE41}) (with $\rho=1$),
\begin{equation}\label{CREN1}
\sup_n\mathbb{E}|W_n-1|^p\leq C\sup_n\mathbb{E}W_n^{p/2}\left(\sum_{n=0}^{\infty}P_n^{-1}(\mathbb{E}|\bar{X}_n-1|^p)^{2/p}\right)^{p/2}<\infty.
\end{equation}
So (\ref{CRE1.4.9}) holds for $b=1$. Now assume  that (\ref{CRE1.4.9}) holds for $p\in(2^b,2^{b+1}]$ for some integer $b\geq1$. For $p\in(2^{b+1},2^{b+2}]$,
we have $p/2\in(2^b,2^{b+1}]$. By H\"older's inequality,
$$\sum_nP_n^{-1}(\mathbb{E}|\bar{X}_n-1|^{p/2})^{4/p}\leq\sum_nP_n^{-1}(\mathbb{E}|\bar{X}_n-1|^p)^{2/p}<\infty.$$
Using (\ref{CRE1.4.9}) for $p/2$, we obtain $\sup_n\mathbb{E}W_n^{p/2}<\infty$, so that $\sup_n\mathbb{E}W_n^{p}<\infty$ from (\ref{CREN1}).
Therefore (\ref{CRE1.4.9}) still holds for $p\in(2^{b+1},2^{b+2}]$, which implies that (\ref{CRE1.4.9}) holds for all integers $b\geq1$.

Conversely, assume that $\sup_n\mathbb{E}|\hat{A}_n|^p<\infty$. Notice that by (\ref{CRE1.3.4} ) and
(\ref{CRE1.4.3}), $\forall r\in[2,p]$,
$$\sup_n\mathbb{E}|\hat{A}_n|^p\geq C\sum_{n=0}^{\infty}\rho^{pn}P_n^{p(1/r-1)}(\mathbb{E}|\bar{X}_n-1|^r)^{p/r}.$$
This implies that $\sum_n\rho^{pn}P_n^{p(1/r-1)}
\bar{m}_n(r)^{p/r}<\infty$, $\forall r\in[2,p]$.
\end{proof}

Now we give proofs of Theorems \ref{CRT1.3} and \ref{CRC1.5}.

\begin{proof}[Proof of Theorem \ref{CRT1.3}]
Notice that $W_{n+1}-1=\hat A_n(1)$. Since $W_n\rightarrow W$ in $L^p$ is equivalent to $\sup_n \mathbb E W_n^p<\infty$, Theorem \ref{CRT1.3} is just a consequence of Proposition \ref{CRP1.4.1} with $\rho=1$.
\end{proof}

\begin{proof}[Proof of Theorem \ref{CRC1.5}]
Notice that the assertion  $\sup_n\mathbb{E}|\hat{A}_n|^p<\infty$ is  equivalent to the $L^p$ convergence of $\hat A(\rho)$, which is also equivalent to the $L^p$ convergence of $A(\rho)$  by Lemma \ref{CRL1.3.2}. So Theorem \ref{CRC1.5} is just a consequence of Proposition \ref{CRP1.4.1} with $\rho>1$.
\end{proof}

\section{Quenched moments and quenched $L^p$ convergence rate for BPRE; Proofs of Theorems \ref{CRT2.1.3} and \ref{CRC2.1.5}}\label{CRLPS5}
Let us return to  a BPRE $(Z_n)$. Notice that for each  fixed $\xi$, $(Z_n)$ is a BPVE. So all the results for
BPVE can be directly applied to BPRE by considering the quenched law $\mathbb{P}_\xi$ and the corresponding expectation $\mathbb{E}_\xi$.
The following lemma will be used to prove our theorems for BPRE.

\begin{lem}\label{CRL2.2.1}
Let $(\alpha_n,\beta_n)_{n\geq0}$ be a stationary and ergodic sequence of non-negative random variables. If $\mathbb{E}\log\alpha_0<0$ and  $\mathbb{E}\log^+\beta_0<\infty$, then
\begin{equation}\label{CRE2.2.1}
\sum_{n=0}^{\infty}\alpha_0\cdots\alpha_{n-1}\beta_n<\infty\quad a.s..
\end{equation}
Conversely, we have:
\begin{itemize}
\item[] (a)  if $(\alpha_n,\beta_n)_{n\geq0}$ are $i.i.d.$ and $\mathbb{E}\log\alpha_0\in(-\infty,0)$, then (\ref{CRE2.2.1}) implies that $\mathbb{E}\log^+\beta_0<\infty$;

 \item[](b) if $\mathbb{E}|\log\beta_0|<\infty$, then (\ref{CRE2.2.1}) implies that $\mathbb{E}\log\alpha_0\leq0$.
\end{itemize}
\end{lem}

\begin{proof}
The sufficiency is a direct consequence of the ergodic theorem and Cauchy's test for the convergence of series, remarking that if $\mathbb{E}\log\alpha_0<0$ and  $\mathbb{E}\log\max(\beta_0,1)<\infty$, then
$$\limsup_{n\rightarrow\infty}\frac{1}{n}\log (\alpha_0\cdots\alpha_{n-1}\max(\beta_n,1))<0.$$

For the necessity, part (a) was shown in the proof of (\cite{grin},  Theorem 4.1). For part (b), again by  Cauchy's test,
 if (\ref{CRE2.2.1}) holds, then
$$\limsup_{n\rightarrow\infty}(\alpha_0\cdots\alpha_{n-1}\beta_n)^{1/n}\leq1 \quad a.s.,$$
which is equivalent to $$\limsup_{n\rightarrow\infty}(\frac{1}{n}\sum_{i=0}^{n-1}\log\alpha_i+\frac{1}{n}\log \beta_n)\leq0\quad a.s..$$
By the ergodic theorem,
$$\lim_{n\rightarrow\infty}\frac{1}{n}\sum_{i=0}^{n-1}\log\alpha_i=\mathbb{E}\log\alpha_0\qquad a.s.,$$
and
$$\lim_{n\rightarrow\infty}\frac{1}{n}\log \beta_n=\lim_{n\rightarrow\infty}\left(\frac{1}{n}\sum_{i=0}^{n}\log\beta_i-
\frac{1}{n}\sum_{i=0}^{n-1}\log\beta_i\right)=\mathbb{E}\log\beta_0-\mathbb{E}\log\beta_0=0\quad a.s..$$
Hence $\mathbb{E}\log\alpha_0\leq0$.
\end{proof}

\begin{proof}[Proof of Theorem \ref{CRT2.1.3}]
The implications  "(ii) $\Rightarrow$ (iii) $\Rightarrow$ (iv)" are evident.
We first prove that (iv) implies (ii). Notice that for $n\geq1$,
\begin{equation}\label{CRLPEE5.1}
W=\frac{1}{P_n}\sum_{i=1}^{Z_n}W(n,i)\qquad a.s.,
\end{equation}
where under $\mathbb{P}_\xi$, $(W(n,i))_{i\geq1}$ are independent of each other and independent of $Z_n$, with distribution $\mathbb{P}_\xi(W(n,i)\in \cdot)=\mathbb{P}_{T^n\xi}(W\in\cdot)$. Taking conditional expectation  at both sides of (\ref{CRLPEE5.1}), we see that
$$\mathbb{E}_\xi W=\mathbb{E}_{T^n\xi} W\qquad a.s..$$
Therefore, by the ergodicity, $\mathbb{E}_\xi W=c$ a.s. for some constant $c\in[0,\infty]$. As $\mathbb{E}_\xi W^p>0$ a.s., we have $c>0$. Again by (\ref{CRLPEE5.1}) and Jensen's inequality,
$$\mathbb{E}_\xi(W^p|\mathcal{F}_n)\geq\left(\mathbb{E}_\xi\left(\left.\frac{1}{P_n}\sum_{i=1}^{Z_n}W(n,i)\right|\mathcal{F}_n\right)\right)^p=c^pW_n^p\qquad a.s.,$$
so that $$\mathbb{E}_\xi W_n^p\leq c^{-p}\mathbb{E}_\xi W^p\qquad a.s. , \quad \forall n\geq1.$$
Therefore, $\sup_n\mathbb{E}_\xi W_n^p\leq c^{-p}\mathbb{E}_\xi W^p<\infty$ a.s. (so that $c=1$ as then $W_n\rightarrow W$ in $L^p$ under $\mathbb{P}_\xi$).

We next prove that (i) implies (ii). Notice that $\mathbb{E}\log \mathbb{E}_\xi\left(\frac{Z_1}{m_0}\right)^p<\infty$ is equivalent to $\mathbb{E}\log^+\mathbb{E}_\xi|\frac{Z_1}{m_0}-1|^p<\infty$.
By Theorem \ref{CRT1.3}, to prove that $\sup_n\mathbb{E}_\xi W_n^p<\infty\;a.s.$,
it suffices to show that
$$\sum_nP_n^{1-p} \bar{m}_n(p)<\infty \;a.s.\quad if\; p\in(1,2),$$
or
$$\sum_nP_n^{-1} \bar{m}_n(p)^{2/p}<\infty \;a.s. \quad if\; p\geq2.$$
By Lemma \ref{CRL2.2.1}, since $\mathbb{E}\log m_0>0$ and $\mathbb{E}\log^+\bar{m}_0(p)=\mathbb{E}\log^+\mathbb{E}_\xi|\frac{Z_1}{m_0}-1|^p<\infty$,
the two series above converge a.s..

We finally prove that  (ii) implies (i) when the environment is i.i.d..  Assume that $(\xi_n)_{n\geq0}$ are i.i.d,    $\mathbb{E}\log m_0<\infty$ and $\sup_n\mathbb{E}_\xi W_n^p<\infty$ a.s.. Again by Theorem \ref{CRT1.3},
$$\sum_nP_n^{-s-p/2}\bar{m}_n(p)<\infty \;a.s., \forall s>0, \quad if\; p\in(1,2),$$
and
$$\sum_nP_n^{1-p}\bar{m}_n(p)<\infty \;a.s.\quad if\; p\geq2.$$
As $(\xi_n)_{n\geq0}$ are $i.i.d.$ and $\mathbb{E}\log m_0\in(0,\infty)$, by Lemma \ref{CRL2.2.1},
$\mathbb{E}\log^+\mathbb{E}_\xi|\frac{Z_1}{m_0}-1|^p<\infty$, so that $\mathbb{E}\log
\mathbb{E}_\xi\left(\frac{Z_1}{m_0}\right)^p<\infty$.
\end{proof}

\begin{pr}[Quenched moments of $\hat A_n$]\label{CRP2.3.1}
Let  $\rho\geq1$ and $m=\exp({\mathbb{E}\log m_0})>1$.
\begin{itemize}
\item[](i) Let $p\in(1,2)$. If $\mathbb{E}\log^+\mathbb{E}_\xi|\frac{Z_1}{m_0}-1|^r<\infty$ and $\rho< m^{1-1/r} $ for some $r\in[p,2]$,
 then
\begin{equation}\label{CREE5.1}
\sup_n\mathbb{E}_\xi|\hat{A}_n|^p<\infty\quad a.s..
\end{equation}
Conversely, if $\mathbb{E}\left|\log \mathbb{E}_\xi|\frac{Z_1}{m_0}-1|^p\right|<\infty$ and (\ref{CREE5.1}) holds,  then $\rho\leq m^{1/2}$.

\item[](ii) Let $p\geq2$. If $\mathbb{E}\log^+\mathbb{E}_\xi|\frac{Z_1}{m_0}-1|^p<\infty$ and $\rho< m^{1/2}$, then (\ref{CREE5.1}) holds. Conversely, if $\mathbb{E}\left|\log \mathbb{E}_\xi|\frac{Z_1}{m_0}-1|^r\right|<\infty$ for some $r\in[2,p]$ and (\ref{CREE5.1}) holds, then $\rho\leq m^{1-1/r} $.
\end{itemize}
\end{pr}

\begin{proof}
 (i) Let $p\in(1,2)$. Suppose that $\mathbb{E}\log^+\mathbb{E}_\xi|\frac{Z_1}{m_0}-1|^r<\infty$ and $\rho< m^{1-1/r} $ for some $r\in[p,2]$.
Then by Lemma \ref{CRE2.2.1}, the series $\sum_n\rho^{pn}P_n^{p(1/r-1)} \bar{m}_n(r)^{p/r}<\infty$ a.s.. Thus $\sup_n\mathbb{E}_\xi|\hat{A}_n|^p<\infty$ a.s.
by Proposition \ref{CRP1.4.1}.

Conversely, suppose that $\mathbb{E}\left|\log \mathbb{E}_\xi|\frac{Z_1}{m_0}-1|^p\right|<\infty$ and $\sup_n\mathbb{E}_\xi|\hat{A}_n|^p<\infty$ a.s.. By Proposition \ref{CRP1.4.1},
we have $\forall s>0$, $\sum_n\rho^{pn}P_n^{-s-p/2} \bar{m}_n(p)<\infty$ a.s.. Hence by Lemma \ref{CRE2.2.1},
 $\rho\leq m^{1/2+s/p}$. Letting $s\rightarrow0$, we get $\rho\leq m^{1/2}$.

(ii) Let $p\geq2$. Suppose that $\mathbb{E}\log^+\mathbb{E}_\xi|\frac{Z_1}{m_0}-1|^p<\infty$ and $\rho< m^{1/2}$. Then by Lemma \ref{CRE2.2.1}, the series
$\sum_n\rho^{2n}P_n^{-1}\bar{m}_n(p)^{2/p}<\infty\;a.s.$, which implies that $\sup_n\mathbb{E}_\xi|\hat{A}_n|^p<\infty$ a.s. by Proposition \ref{CRP1.4.1}.

Conversely, suppose that $\mathbb{E}\left|\log \mathbb{E}_\xi|\frac{Z_1}{m_0}-1|^r\right|<\infty$ for some $r\in[2,p]$ and $\sup_n\mathbb{E}_\xi|\hat{A}_n|^p<\infty$ a.s.. Proposition \ref{CRP1.4.1} shows that $\sum_n\rho^{pn}P_n^{p(1/r-1)} \bar{m}_n(r)^{p/r}<\infty$ a.s.,
which implies that $\rho\leq m^{1-1/r} $ by Lemma \ref{CRE2.2.1}.
\end{proof}

 By  the relations of $\hat A_n$, $A(\rho)$ and $\rho^n(W-W_n)$ (discussed at the beginning of Section \ref{CRLPS3}), together with  Proposition \ref{CRP2.3.1},  we immediately obtain  the following criteria for the  quenched  $L^p$  convergence rate of $W_n$.

\begin{thm}[Exponential  rate of quenched $L^p$ convergence of $W_n$]\label{CRC5.5}  Let $\rho>1$ and $m=\exp({\mathbb{E}\log m_0})>1$.
\begin{itemize}
\item[] (i) Let $p\in(1,2)$. If $\mathbb{E}\log
\mathbb{E}_\xi\left(\frac{Z_1}{m_0}\right)^r<\infty$ and $ \rho<m^{1-1/r }$
for some $r\in[p,2]$, then
\begin{equation}\label{CRE2.1.1}
(\mathbb{E}_\xi|W-W_n|^p)^{1/p}=o(\rho^{-n})\qquad a.s..
\end{equation}
 Conversely, if $\mathbb{E}\left|\log
\mathbb{E}_\xi\left|\frac{Z_1}{m_0}-1\right|^p\right|<\infty$ and
(\ref{CRE2.1.1}) holds, then $\rho\leq m^{1/2}$.
\item[] (ii) Let $p\geq2$. If $\mathbb{E}\log \mathbb{E}_\xi\left(\frac{Z_1}{m_0}\right)^p<\infty$
and $\rho<m^{1/2}$, then (\ref{CRE2.1.1}) holds. Conversely, if
$\mathbb{E}\left|\log \mathbb{E}_\xi\left|\frac{Z_1}{m_0}-1\right|^r\right|<\infty$
for some $r\in[2,p]$ and (\ref{CRE2.1.1}) holds, then
$\rho\leq m^{1-1/r}$.
\end{itemize}
\end{thm}

\begin{proof}[Proof of  Theorem \ref{CRC2.1.5}]
The assertion (a) is a direct consequence of Theorem \ref{CRC5.5}(i) with $r=p$ for $p\in(1,2)$  and  Theorem \ref{CRC5.5}(ii) for $p\geq 2$.

For the assertion (b), notice that the condition $\mathbb{E}\log^+
\mathbb{E}_\xi\left|\frac{Z_1}{m_0}-1\right|^{p\vee2}<\infty$ ensures that $\mathbb{E}\log^+
\mathbb{E}_\xi\left|\frac{Z_1}{m_0}-1\right|^{p}<\infty$ and $\mathbb{E}\log^+
\mathbb{E}_\xi\left|\frac{Z_1}{m_0}-1\right|^{2}<\infty$. If $\rho<m^{1/2}$, applying Theorem \ref{CRC5.5}(i) with $r=2$ for $p\in(1,2)$  and  Theorem \ref{CRC5.5}(ii) for $p\geq 2$, we have $$\lim_{n\rightarrow\infty}\rho^n(\mathbb{E}_\xi|W-W_n|^p)^{1/p}=0\qquad a.s..$$

Now consider the case where $\rho>m^{1/2}$. Denote
$$D=\{\xi: \lim_{n\rightarrow\infty}\rho^n(\mathbb{E}_\xi|W-W_n|^p)^{1/p}=0\}.$$
First, we show that $\mathbb{P}(D)=0$ or $1$. By the ergodicity, it suffices to show that $T^{-1}D=D$ a.s.. By (\ref{CRLPEE5.1}),
$$W=\frac{1}{m_0}\sum_{i=1}^{Z_1}W(1,i)\qquad a.s..$$
Similarly, we can write $W_n$ as
\begin{equation}\label{CRE2.4.7}
W_n=\frac{1}{m_0}\sum_{i=0}^{Z_1}W_{n-1}(1,i)\qquad a.s.,
\end{equation}
where $W_{n}{(k,i)}=\frac{Z_{n}{(k,i)}}{m_k\cdots m_{k+n-1}}$ with
$Z_{n}{(k,i)}$ denoting the branching process starting with the $i$th
particle in the $k$th generation. Under $\mathbb{P}_\xi$, the sequence
$(W_{n}{(k,i)})_{i\geq1}$ are independent of each other and independent of $Z_k$, and have a common
conditional distribution
$\mathbb{P}_\xi(W_{n}{(k,i)}\in\cdot)=\mathbb{P}_{T^k\xi}(W_n\in\cdot)$. Therefore,
\begin{equation}\label{CRLPEE52}
W-W_n=\frac{1}{m_0}\sum_{i=1}^{Z_1}\left(W(1,i)-W_{n-1}(1,i)\right)\qquad a.s..
\end{equation}
By (\ref{CRLPEE52}) and the convexity of $x^p$, we have
\begin{eqnarray}
\mathbb{E}_\xi|W-W_n|^p&\leq&\frac{1}{m_0^p}\mathbb{E}_\xi\left(\sum_{i=1}^{Z_1}|W(1,i)-W_{n-1}(1,i)|\right)^p\nonumber\\
&\leq&\frac{1}{m_0^p}\mathbb{E}_\xi Z_1^{p-1}\sum_{i=1}^{Z_1}|W(1,i)-W_{n-1}(1,i)|^p\nonumber\\
&=&\mathbb{E}_\xi\left(\frac{Z_1}{m_0}\right)^p\mathbb{E}_{T\xi}|W-W_{n-1}|^p.
\label{CRLPEE53}
\end{eqnarray}
Therefore for almost all $\xi$, if $T\xi\in D$, then  $\xi\in D$. So we have proved that $T^{-1}D\subset D$ a.s.. On the other hand, notice that by Theorem \ref{CRT2.1.3}, $\mathbb{E}_\xi W=1$ a.s.. Using (\ref{CRLPEE52}) and Burkholder's inequality, we get
\begin{eqnarray}
\mathbb{E}_\xi|W-W_n|^p&\geq&\frac{C}{m_0^p}\mathbb{E}_\xi\left(\sum_{i=1}^{Z_1}(W(1,i)-W_{n-1}(1,i))^2\right)^{p/2}\nonumber\\
&\geq&\frac{C}{m_0^p}\mathbb{E}_\xi \mathbf{1}_{\{Z_1\geq1\}}\sum_{i=1}^{Z_1}|W(1,i)-W_{n-1}(1,i)|^p\nonumber\\
&=&C \frac{1-p_0(\xi_0)}{m_0^p}\mathbb{E}_{T\xi}|W-W_{n-1}|^p\qquad a.s..
\label{CRLPEE54}
\end{eqnarray}
Notice that $p_0(\xi_0)<1$ since $m_0\in(0,\infty)$. It follows from (\ref{CRLPEE54}) that for almost all $\xi$, if $\xi\in D$,   then $T\xi\in D$. Hence $D\subset T^{-1}D$ a.s.. So we have proved that $T^{-1}D=D$ a.s..

 For $\rho>m^{1/2}$, assume that $\mathbb{P}(D)=1$, so that $\lim_{n\rightarrow\infty}\rho^n(\mathbb{E}_\xi|W-W_n|^p)^{1/p}=0$ a.s..  Notice that the condition $\mathbb{E}\log^-\mathbb{E}_\xi\left|\frac{Z_1}{m_0}-1\right|^{p\wedge2}<\infty$ ensures that $\mathbb{E}\log^-
\mathbb{E}_\xi\left|\frac{Z_1}{m_0}-1\right|^{p}<\infty$ and $\mathbb{E}\log^-
\mathbb{E}_\xi\left|\frac{Z_1}{m_0}-1\right|^{2}<\infty$. So we have $\mathbb{E}\left|\log
\mathbb{E}_\xi\left|\frac{Z_1}{m_0}-1\right|^{p}\right|<\infty$ and $\mathbb{E}\left|\log
\mathbb{E}_\xi\left|\frac{Z_1}{m_0}-1\right|^{2}\right|<\infty$.
Applying Theorem \ref{CRC5.5}(i) for $p\in(1,2)$ and  Theorem \ref{CRC5.5}(ii) with $r=2$ for $p\geq 2$, we get $\rho\leq m^{1/2}$. This contradicts the condition that $\rho>m^{1/2}$. Thus $\mathbb{P}(D)=0$, which implies that

$$\mathbb{P}\left(\limsup_{n\rightarrow\infty}\rho^n(\mathbb{E}_\xi|W-W_n|^p)^{1/p}>0\right)=\mathbb{P}(D^c)=1.$$
So the proof is finished.
\end{proof}

\section{Annealed moments and annealed $L^p$ convergence rate for BPRE; Proof of Theorem \ref{CRC5.9}}\label{CRLPS6}
In this section, we  consider a branching process in an $i.i.d.$ environment: we assume that $(\xi_n)_{n\geq0}$ are $i.i.d.$. We also assume  that
\begin{equation}
\mathbb{P}(W_1=1)<1,
\end{equation}
which avoids the trivial case where $W_n=1$ a.s..
\\*

 Let us study the annealed moments of  $ {\hat A_n}$ at first. We shall distinguish two cases: (i) $p\in(1,2)$; (ii) $p\geq2$. Our approach
is inspired by ideas from \cite{IK} and \cite{liu1}, especially for the case where $p\geq2$.

\subsection{Annealed moments of $\hat A_n$: case $p\geq 2$}
We first consider the case where $p\geq2$.
\begin{pr}[Annealed moments of $\hat A_n$ for $p\geq2$]\label{CRP2.4.3}
Let $p\geq2$ and $\rho\geq1$. Then $\sup_n\mathbb{E}|\hat{A}_n|^p<\infty$ if and only if
$\mathbb{E}\left(\frac{Z_1}{m_0}\right)^p<\infty$ and
 $\rho\max\{(\mathbb{E}m_0^{1-p})^{1/p}$,
 $(\mathbb{E}m_0^{-p/2})^{1/p}\}<1$.
\end{pr}

To prove Proposition \ref{CRP2.4.3} for $p>2$, we need two lemmas below.
Denote
\begin{equation}
u_n(s,r)=\mathbb{E}P_n^{-s}W_n^r\quad(s\in\mathbb{R},r>1).
\end{equation}

\begin{lem}\label{CRL2.4.1}
For $r>2$, $u_n(s,r)$ satisfies the following recursive formula:
\begin{equation}\label{CRE2.4.6}
u_n(s,r)^{\frac{1}{r-1}}\leq
(\mathbb{E}m_0^{1-r-s})^{\frac{1}{r-1}}u_{n-1}(s,r)^{\frac{1}{r-1}}
+(\mathbb{E}m_0^{-s}W_1^r)^{\frac{1}{r-1}}u_{n-1}(s,r-1)^{\frac{1}{r-1}}.
\end{equation}
\end{lem}

\begin{proof}
Denote $\varphi_n^{(k)}(t)=\mathbb{E}_{T^k\xi}e^{itW_n}$ and $\varphi_n(t)=\varphi_n^{(0)}(t)
=\mathbb{E}_{\xi}e^{itW_n}$.  By
(\ref{CRE2.4.7}), we  get the functional equation
$$\varphi_n(s) =\mathbb{E}_\xi\varphi_{n-1}^{(1)} (\frac{t}{m_0} )^{Z_1}\qquad a.s..$$
By differentiations, this yields
\begin{equation}\label{CRE2.4.8}
\varphi'_n(t)=\mathbb{E}_\xi\frac{Z_1}{m_0}\left(\varphi_{n-1}^{(1)}
(\frac{t}{m_0} )\right)^{Z_1-1}\left(\varphi_{n-1}^{(1)}
(\frac{t}{m_0} )\right)'\qquad a.s..
\end{equation}
Let $V_n$ be a random variable whose distribution is determined by
$$\mathbb{E}_\xi g(V_n)=\mathbb{E}_\xi W_ng(W_n)$$
for all bounded and measurable function $g$, and  $V_n^{(k)}$ with
$\mathbb{P}_\xi(V_n^{(k)}\in\cdot)=\mathbb{P}_{T^k\xi}(V_n\in\cdot)$. Let $M_n$ be a
random variable  independent of $V_{n-1}^{(1)}$ under $\mathbb{P}_\xi$, whose distribution  is determined by
$$\mathbb{E}_\xi g(M_n)=\mathbb{E}_\xi\frac{Z_1}{m_0}g(\frac{1}{m_0}\sum_{i=0}^{Z_1-1}W_{n-1}{(1,i)}),$$
for all  bounded and measurable function $g$. (The probability space ($\Gamma, \mathbb{P}_\xi$) can be taken large enough to define the random variables $V_n$, $V_n^{(k)}$ and $M_n$.)
The Fourier transform of $V_n$ is
$$\mathbb{E}_\xi e^{itV_n}=\mathbb{E}_\xi W_ne^{itW_n}=-i\varphi'_n(t).$$
So (\ref{CRE2.4.8}) implies that
$$\mathbb{E}_\xi e^{itV_n}=\mathbb{E}_\xi e^{it(\frac{1}{m_0}V_{n-1}^{(1)}+M_n)}\qquad a.s.,$$
which is equivalent to the distributional equation
$$V_n\stackrel{d}{=}\frac{1}{m_0}V_{n-1}^{(1)}+M_n$$
under $ \mathbb{P}_\xi$. Therefore,
\begin{eqnarray*}
u_n(s,r)=\mathbb{E}P_n^{-s}W_n^r&=&\mathbb{E}P_n^{-s}\mathbb{E}_\xi W_n^r=\mathbb{E}P_n^{-s}\mathbb{E}_\xi V_n^{r-1}\\
&=&\mathbb{E}P_n^{-s}\mathbb{E}_\xi \left(\frac{1}{m_0}V_{n-1}^{(1)}+M_n\right)^{r-1}\\
&=&\mathbb{E}\left(P_n^{-\frac{s}{r-1}}m_0^{-1}V_{n-1}^{(1)}+P_n^{-\frac{s}{r-1}}M_n\right)^{r-1}.
\end{eqnarray*}
By the triangular inequality in $L^{r-1}$,
\begin{equation}\label{CRE2.4.9}
u_n(s,r)^{\frac{1}{r-1}}\leq
\left(\mathbb{E}P_n^{-s}m_0^{1-r}\left(V_{n-1}^{(1)}\right)^{r-1}\right)^{\frac{1}{r-1}}
+\left(\mathbb{E}P_n^{-s}M_n^{r-1}\right)^{\frac{1}{r-1}}.
\end{equation}
We now calculate the two expectations of the right hand side. We have
\begin{eqnarray}\label{CRE2.4.10}
\mathbb{E}P_n^{-s}m_0^{1-r}\left(V_{n-1}^{(1)}\right)^{r-1}&=&\mathbb{E}P_n^{-s}m_0^{1-r}\mathbb{E}_{T\xi}V_{n-1}^{r-1}\nonumber\\
&=&\mathbb{E}m_0^{1-r-s}\mathbb{E}P_n^{-s}V_{n-1}^{r-1}\nonumber\\
&=&\mathbb{E}m_0^{1-r-s}u_{n-1}(s,r),
\end{eqnarray}
and
\begin{eqnarray}\label{CRE2.4.11}
\mathbb{E}P_n^{-s}M_n^{r-1}&=&\mathbb{E}P_n^{-s}\mathbb{E}_\xi M_n^{r-1}\nonumber\\
&=&\mathbb{E}P_n^{-s}\mathbb{E}_\xi\frac{Z_1}{m_0}\left(\frac{1}{m_0}\sum_{i=0}^{Z_1-1}W_{n-1}{(1,i)}\right)^{r-1}\nonumber\\
&\leq&\mathbb{E}P_n^{-s}m_0^{-r}\mathbb{E}_\xi Z_1^{r}\mathbb{E}_{T\xi}W_{n-1}^{r-1}\nonumber\\
&=&\mathbb{E}m_0^{-s} \left(\frac{Z_1}{m_0} \right)^r\mathbb{E}P_{n-1}^{-s}W_{n-1}^{r-1}\nonumber\\
&=&\mathbb{E}m_0^{-s} \left(\frac{Z_1}{m_0} \right)^ru_{n-1}(s,r-1).
\end{eqnarray}
So (\ref{CRE2.4.6}) is a combination of
(\ref{CRE2.4.9}), (\ref{CRE2.4.10}) and (\ref{CRE2.4.11}).
\end{proof}

\noindent\textbf{Remark}. In particular, $u_n(0,r)=\mathbb{E}W_n^r$. By Lemma
\ref{CRL2.4.1}, we can obtain the recursive formula for $\mathbb{E}W_n^r$:
$$(\mathbb{E}W_n^r)^{\frac{1}{r-1}}\leq
(\mathbb{E}m_0^{r-1})^{\frac{1}{r-1}}+(\mathbb{E}W_1^r)^{\frac{1}{r-1}}(\mathbb{E}W_{n-1}^{r-1})^{\frac{1}{r-1}}\quad(r>2).$$
 \\*

\begin{lem}\label{CRL2.4.2} Let $s\in\mathbb{R}$ and $r\in(b,b+1]$, where $b\geq 1$ is an integer.
If $\mathbb{E}m_0^{-s}<\infty$ and
$\mathbb{E}m_0^{-s}\left(\frac{Z_1}{m_0}\right)^r<\infty$, then
$$u_n(s,r)=O(n^{1+(b-1)r-(b-1)b/2}(\max\{\max_{1\leq i\leq b}\mathbb{E}m_0^{i-r-s},\;\mathbb{E}m_0^{-s}\})^n).$$
\end{lem}

\begin{proof}
We shall prove this lemma by induction on $b$. For $b=1$, let
$r\in(1,2]$. By Burkholder's inequality,
\begin{eqnarray*}
\mathbb{E}_\xi W_n^r&\leq&1+\sup_n\mathbb{E}_\xi|W_n-1|^r\\
&\leq&1+C\sum_{k=0}^{n-1}P_k^{1-r}\mathbb{E}_\xi|\bar{X}_n-1|^r\quad a.s..
\end{eqnarray*}
Hence
\begin{eqnarray*}
u_n(s,r)&=&\mathbb{E}P_n^{-s}\mathbb{E}_\xi W_n^r\\
&\leq&\mathbb{E}P_n^{-s}\left(1+C\sum_{k=0}^{n-1}P_k^{1-r}\mathbb{E}_\xi|\bar{X}_n-1|^r\right)\\
&=&\mathbb{E}P_n^{-s}+C\sum_{k=0}^{n-1}\mathbb{E}P_n^{-s}P_k^{1-r}\mathbb{E}_\xi|\bar{X}_n-1|^r\\
&=&(\mathbb{E}m_0^{-s})^n+C\sum_{k=0}^{n-1}(\mathbb{E}m_0^{1-r-s})^{k}(\mathbb{E}m_0^{-s})^{n-k-1}\mathbb{E}m_0^{-s}|\bar{X}_0-1|^r\\
&\leq&(\mathbb{E}m_0^{-s})^n+Cn\max\{\mathbb{E}m_0^{1-r-s},\;\mathbb{E}m_0^{-s}\}^{n-1}\\
&=&O(n(\max\{\mathbb{E}m_0^{1-r-s},\;\mathbb{E}m_0^{-s}\})^{n}).
\end{eqnarray*}
So the conclusion holds for $b=1$.

Now we assume that the conclusion is true for $r\in(b,b+1]$ for some integer $b\geq1$. Then
for $r\in(b+1,b+2]$, $r-1\in(b,b+1]$. By H\"older's inequality,
\begin{eqnarray*}
\mathbb{E}m_0^{-s}\left(\frac{Z_1}{m_0}\right)^{r-1}&=&\mathbb{E}m_0^{-s/r}m_0^{-s(r-1)/r}\left(\frac{Z_1}{m_0}\right)^{r-1}\\
&\leq&(\mathbb{E}m_0^{-s})^{1/r}\left(\mathbb{E}m_0^{-s}\left(\frac{Z_1}{m_0}\right)^{r}\right)^{(r-1)/r},
\end{eqnarray*}
which implies that $\mathbb{E}m_0^{-s}(\frac{Z_1}{m_0})^{r-1}<\infty$, since
$\mathbb{E}m_0^{-s}<\infty$ and $\mathbb{E}m_0^{-s}(\frac{Z_1}{m_0})^{r}<\infty$. By
the induction assumption,
\begin{eqnarray}\label{CRE2.4.12}
u_{n-1}(s,r-1)&=&O((n-1)^{1+(b-1)(r-1)-(b-1)b/2}(\max\{\max_{1\leq i\leq b}\mathbb{E}m_0^{i+1-r-s},\;\mathbb{E}m_0^{-s}\})^{n-1})\nonumber\\
&=&O(n^{1+(b-1)(r-1)-(b-1)b/2}(\max\{\max_{2\leq i\leq
b+1}\mathbb{E}m_0^{i-r-s},\;\mathbb{E}m_0^{-s}\})^{n}).
\end{eqnarray}
It is easy to verify that any solution to the recursive inequality
\begin{equation}\label{CRE2.4.13}
c_n\leq\alpha
c_{n-1}+O(n^\gamma\beta^n)\quad(\alpha,\beta,\gamma\geq0)
\end{equation}
satisfies $c_n=O(n^{\gamma+1}\max\{\alpha,\beta\}^n)$. Lemma
\ref{CRL2.4.1} and (\ref{CRE2.4.12}) show that
$u_n(s,r)^{\frac{1}{r-1}}$ is a solution of (\ref{CRE2.4.13}) with
$\alpha=(\mathbb{E}m_0^{1-r-s})^{\frac{1}{r-1}}$, $\beta=\max\{\max_{2\leq i\leq
b+1}(\mathbb{E}m_0^{i-r-s})^{\frac{1}{r-1}}, (\mathbb{E}m_0^{-s})^{\frac{1}{r-1}}\}$
and $\gamma=\frac{1+(b-1)(r-1)-(b-1)b/2}{r-1}$. Thus
\begin{equation}\label{CRE2.4.14}
u_n(s,r)^{\frac{1}{r-1}}=O(n^{\gamma+1}\max\{\alpha,\beta\}^n).
\end{equation}
Notice that $\gamma+1=\frac{1+ b r-b(b+1)/2}{r-1} $ and
\begin{eqnarray*}
\max\{\alpha,\beta\}=
\max\{\max_{1\leq i\leq
b+1}(\mathbb{E}m_0^{i-r-s})^{\frac{1}{r-1}},(\mathbb{E}m_0^{-s})^{\frac{1}{r-1}}\}.
\end{eqnarray*}
Hence (\ref{CRE2.4.14}) becomes
$$u_n(s,r)=O(n^{1+ b r-b(b+1)/2}(\max\{\max_{1\leq i\leq b+1}\mathbb{E}m_0^{i-r-s},  \;\mathbb{E}m_0^{-s} \})^n).$$
So the conclusion still holds for $r\in(b+1,b+2]$. This completes the
proof.
\end{proof}

\noindent\textbf{Remark}. In Lemma \ref{CRL2.4.2},  since
$1-(b-1)b/2\leq0$ for $b\geq2$, we in fact obtain
$$u_n(s,r)=O(n(\max\{\mathbb{E}m_0^{1-r-s}, \;\mathbb{E}m_0^{-s} \})^n)\quad \text{for}\;r\in(1,2],$$
and for any integer $b\geq1$,
$$u_n(s,r)=O(n^{br}(\max\{\max_{1\leq i\leq b+1}\mathbb{E}m_0^{i-r-s}, \;\mathbb{E}m_0^{-s} \})^n)\quad \text{for}\;r\in(b+1,b+2].$$
\\*

\begin{proof}[Proof of Proposition \ref{CRP2.4.3}]
For $p=2$, by Lemmas \ref{CRL1.3.3} and \ref{CRL1.4.1},
\begin{eqnarray}\label{CRE2.4.5}
\sup_n\mathbb{E}|\hat{A}_n|^2
=\sum_{n=0}^{\infty}\rho^{2n}\mathbb{E}(P_n^{-1}\mathbb{E}_\xi|\bar{X}_n-1|^2)=\mathbb{E}|\bar{X}_0-1|^2\sum_{n=0}^{\infty}\left(\rho^{2}\mathbb{E}m_0^{-1}
\right)^n.
\end{eqnarray}
Therefore, $\sup_n\mathbb{E}|\hat{A}_n|^2<\infty$ if and only if
$\mathbb{E}(\frac{Z_1}{m_0})^2<\infty$ and $\rho(\mathbb{E}m_0^{-1})^{1/2}<1$.

Now we consider the case where $p>2$.
Assume that $\mathbb{E}\left(\frac{Z_1}{m_0}\right)^p<\infty$ and
$\rho\max\{(\mathbb{E}m_0^{1-p})^{1/p}$,  $(\mathbb{E}m_0^{-p/2})^{1/p}\}$ $<1$. By
Lemma \ref{CRL1.3.3},
$$\sup_n\mathbb{E}|\hat{A}_n|^p\leq C\left(\sum_{n=0}^{\infty}\rho^{2n}(\mathbb{E}|W_{n+1}-W_n|^p)^{2/p}\right)^{p/2}.$$
To prove $\sup_n\mathbb{E}|\hat{A}_n|^p<\infty$, it suffices to show that
$$\sum_{n=0}^{\infty}\rho^{2n}(\mathbb{E}|W_{n+1}-W_n|^p)^{2/p}<\infty.$$
By Lemma \ref{CRL1.4.1},
\begin{eqnarray*}
\mathbb{E}|W_{n+1}-W_n|^p&\leq&C \mathbb{E}P_n^{-p/2}\mathbb{E}_\xi W_n^{p/2}\mathbb{E}_\xi|\bar{X}_n-1|^p\\
&=&C \mathbb{E}P_n^{-p/2}  W_n^{p/2}\mathbb{E}|\bar{X}_0-1|^p\\
&=&Cu_n(p/2,p/2).
\end{eqnarray*}
Notice that
$$\mathbb{E}m_0^{-p/2}\left(\frac{Z_1}{m_0}\right)^{p/2}=\mathbb{E}m_0^{-p}Z_1^{p/2}\mathbf{1}_{\{Z_1\geq1\}}
\leq\mathbb{E}m_0^{-p}Z_1^{p}\mathbf{1}_{\{Z_1\geq1\}}\leq\mathbb{E}\left(\frac{Z_1}{m_0}\right)^p<\infty,$$
and $\mathbb{E}m_0^{-p/2}<1<\infty$. The
remark after Lemma \ref{CRL2.4.2} shows that
$$u_n(p/2,p/2)=O(n^{\gamma}(\max\{\max_{1\leq i\leq b+1}\mathbb{E}m_0^{i-p},
\mathbb{E}m_0^{-p/2} \})^n)$$ for $p/2\in(b+1,b+2]$ with $\gamma=1$ for $b=0$
and $\gamma=bp/2$ for $b\geq1$. Notice that $\mathbb{E}m_0^x$ is $\log$
convex. Therefore we have
$$\max\{\max_{1\leq i\leq b+1}\mathbb{E}m_0^{i-p}, \;\mathbb{E}m_0^{-p/2} \}
\leq\sup_{1-p\leq x\leq
-p/2}\{\mathbb{E}m_0^x\}=\max\{\mathbb{E}m_0^{1-p}, \;\mathbb{E}m_0^{-p/2}\}.$$ Thus
$$\sum_{n=0}^{\infty}\rho^{2n}(\mathbb{E}|W_{n+1}-W_n|^p)^{2/p}
\leq
C\sum_{n=0}^{\infty}\rho^{2n}n^{2\gamma/p}(\max\{(\mathbb{E}m_0^{1-p})^{2/p},(\mathbb{E}m_0^{-p/2)^{2/p}}\})^n.$$
The series in the right side of the above inequality is finite if and only if
$\rho\max\{(\mathbb{E}m_0^{1-p})^{1/p}$, $(\mathbb{E}m_0^{-p/2})^{1/p}\}$$<1$.

Conversely, assume that $\sup_n\mathbb{E}|\hat{A}_n|^p<\infty$. Obviously,
$\mathbb{E}\left(\frac{Z_1}{m_0}\right)^p<\infty$, since
$\mathbb{E}|\frac{Z_1}{m_0}-1|^p=\mathbb{E}|\hat{A}_0|^p<\infty$. By Lemma
\ref{CRL1.3.3} and Lemma \ref{CRL1.4.1}, we have $\forall r\in[2,p]$,
\begin{eqnarray*}
\sup_n\mathbb{E}|\hat{A}_n|^p&\geq& C\sum_{n=0}^{\infty}\rho^{pn}\mathbb{E}|W_{n+1}-W_n|^p\\
&\geq&C\sum_{n=0}^{\infty}\rho^{pn}\mathbb{E}P_n^{p(1/r-1)}(\mathbb{E}_\xi|\bar{X}_n-1|^r)^{p/r}\\
&=&C\sum_{n=0}^{\infty}\rho^{pn}(\mathbb{E}m_0^{p(1/r-1)})^n\mathbb{E}(\mathbb{E}_\xi|\bar{X}_0-1|^r)^{p/r}.
\end{eqnarray*}
Thus $\rho(\mathbb{E}m_0^{p(1/r-1)})^{1/p}<1$ holds for all $r\in[2,p]$.
Taking $r=p,2$, we get $\rho\max\{(\mathbb{E}m_0^{1-p})^{1/p}$,
$(\mathbb{E}m_0^{-p/2})^{1/p}\}<1$.
\end{proof}

\subsection{Annealed moments of $\hat A_n$: case $p\in(1,2)$}
For the case where $p\in(1,2)$, we have the proposition below.
\begin{pr}[Annealed moments of $\hat A_n$ for $p\in(1, 2)$]\label{CRP2.4.1}
Let $p\in(1,2)$ and $\rho\geq1$.
 If
$\mathbb{E}\left(\mathbb{E}_\xi\left(\frac{Z_1}{m_0}\right)^r\right)^{p/r}<\infty$ and $\rho(\mathbb{E}m_0^{p(1/r-1)})^{1/p}<1$ for some $ r\in[p,2]$,
then
\begin{equation}\label{CREE6.1}
\text{$\sup_n\mathbb{E}|\hat{A}_n|^p<\infty$.}
\end{equation}
Conversely, if (\ref{CREE6.1}) holds, then $\mathbb{E}(\frac{Z_1}{m_0})^p<\infty$ and $\rho(\mathbb{E }m_0^s)^{-1/2s}<1$ for all $ s>0$, so that $\rho\leq \exp({\frac{1}{2}\mathbb{E}\log m_0})$; if additionally
$\mathbb{E}m_0^{-p/2}\log m_0>0$ and $\mathbb{E}m_0^{-p/2-1}Z_1\log^+Z_1<\infty$,
then $\rho(\mathbb{E}m_0^{-p/2})^{1/p}<1$.
\end{pr}

\begin{proof}
Suppose that
$\mathbb{E}\left(\mathbb{E}_\xi\left(\frac{Z_1}{m_0}\right)^r\right)^{p/r}<\infty$ and
$\rho(\mathbb{E}m_0^{p(1/r-1)})^{1/p}<1$ for some $r\in [p, 2]$. By Lemma
\ref{CRL1.4.1},
$$\mathbb{E}_\xi|W_{n+1}-W_n|^p\leq CP_n^{p(1/r-1)}(\mathbb{E}_\xi|\bar{X}_n-1|^r)^{p/r}.$$
Taking expectation we obtain
\begin{equation}\label{CRE2.4.1}
\mathbb{E}|W_{n+1}-W_n|^p\leq C (\mathbb{E}m_0^{p(1/r-1)}
)^n\mathbb{E}(\mathbb{E}_\xi|\bar{X}_0-1|^r)^{p/r}.
\end{equation}
Notice that
$$\mathbb{E}(\mathbb{E}_\xi|\bar{X}_0-1|^r)^{p/r}\leq C\left(\mathbb{E}\left(\mathbb{E}_\xi\left(\frac{Z_1}{m_0} \right)^r\right)^{p/r}+1\right)<\infty.$$
 By Lemma \ref{CRL1.3.3} and (\ref{CRE2.4.1}),
\begin{eqnarray*}
\sup_n\mathbb{E}|\hat{A}_n|^p&\leq& C\sum_{n=0}^{\infty}\rho^{pn}\mathbb{E}|W_{n+1}-W_n|^p\\
&\leq&C\mathbb{E}\left(\mathbb{E}_\xi|\bar{X}_0-1|^r\right)^{p/r}\sum_{n=0}^{\infty}\rho^{pn}(\mathbb{E}m_0^{p(1/r-1)})^n<\infty.
\end{eqnarray*}

Conversely, assume that $\sup_n\mathbb{E}|\hat{A}_n|^p<\infty$. It is
obvious that $\mathbb{E}|\frac{Z_1}{m_0}-1|^p=\mathbb{E}|\hat{A}_0|^p<\infty$. By
Lemmas \ref{CRL1.3.3} and \ref{CRL1.4.1}, we have $\forall N\geq1$,
\begin{eqnarray}\label{CRE2.4.2}
\sup_n\mathbb{E}|\hat{A}_n|^p&\geq& CN^{p/2-1}\sum^{N-1}_{n=0}\rho^{pn}\mathbb{E}|W_{n+1}-W_n|^p\nonumber\\
&\geq&CN^{p/2-1}\sum^{N-1}_{n=0}\rho^{pn}\mathbb{E}P_n^{-p/2}\mathbb{E}_\xi W_n^{p/2}\mathbb{E}_\xi|\bar{X}_n-1|^p\nonumber\\
&=&CN^{p/2-1}\sum^{N-1}_{n=0}\rho^{pn}\mathbb{E}P_n^{-p/2}W_n^{p/2}\mathbb{E}|\bar{X}_0-1|^p.
\end{eqnarray}
The assumption $\mathbb{P}(W_1=1)<1$ ensures that $\mathbb{E}|\bar{X}_0-1|^p>0$.
For  $\alpha>0$, H\"older's inequality gives
\begin{equation}\label{CRE2.4.3}
\mathbb{E}W_n^\alpha=\mathbb{E}W_n^\alpha P_n^{-\alpha}P_n^{\alpha}\leq( \mathbb{E}W_n^{\alpha
p_1} P_n^{-\alpha p_1})^{1/p_1}(\mathbb{E}P_n^{\alpha q_1})^{1/q_1},
\end{equation}
where $p_1,q_1>1$ and $1/p_1+1/q_1=1$.  For $s>0$, take
$\alpha=\frac{sp}{p+2s}$, $p_1=1+p/2s$ and $q_1=1+2s/p$. Then
(\ref{CRE2.4.3}) becomes
\begin{equation}\label{CRE2.4.4}
(\mathbb{E}W_n^\alpha)^{p_1}\leq \mathbb{E}W_n^{p/2}P_n^{-p/2}(\mathbb{E}m_0^s)^{pn/2s}.
\end{equation}
Combing (\ref{CRE2.4.4}) with  (\ref{CRE2.4.2}), we get
\begin{eqnarray*}
\sup_n\mathbb{E}|\hat{A}_n|^p&\geq& CN^{p/2-1}\sum^{N-1}_{n=0}\rho^{pn}
(\mathbb{E}m_0^s)^{-pn/2s}(\mathbb{E}W_n^\alpha)^{p_1}\\
&\geq&
C(\inf_n\mathbb{E}W_n^\alpha)^{p_1}N^{p/2-1}\sum^{N-1}_{n=0}\left(\rho^{p}
(\mathbb{E}m_0^s)^{-p/2s}\right)^n.
\end{eqnarray*}
Hence $\sup_n\mathbb{E}|\hat{A}_n|^p<\infty$ implies that $\rho(E
m_0^s)^{-1/2s}<1$ for all $ s>0$, so that $\log
\rho<\frac{1}{2s}\log \mathbb{E}m_0^s$ for all $ s>0$. Notice that
$(\mathbb{E}m_0^s)^{1/s}$ is increasing as $s$ increases. We have
$$\log \rho\leq\inf_{s>0}\frac{1}{2s}\log \mathbb{E}m_0^s=\frac{1}{2}\lim_{s\rightarrow0^+}\frac{1}{s}\log(\mathbb{E}m_0^s)=\frac{1}{2}\mathbb{E}\log m_0,$$
so that $\rho\leq\exp({\frac{1}{2}\mathbb{E}\log m_0})$.

If additionally $\mathbb{E}m_0^{-p/2}\log m_0>0$ and $\mathbb{E}m_0^{-p/2-1}Z_1\log^+Z_1<\infty$, we introduce a new BPRE. Denote the distribution of $\xi_0$ by $\tau_0$. Define a new distribution $\tilde{\tau}_0$ as
$$\tilde{\tau}_0(dx)=\frac{m(x)^{-p/2}\tau_0(dx)}{\mathbb{E}m_0^{-p/2}},$$
where $m(x)=\mathbb{E}[Z_1|\xi_0=x]=\sum_{k=0}^{\infty}kp_k(x)$.  Consider the new BPRE whose environment
distribution
is $\tilde{\tau}=\tilde{\tau}_0^{\otimes \mathbb N_0}$ instead of $\tau=\tau_0^{\otimes \mathbb N_0}$. The corresponding probability and
expectation are denoted by
$\tilde{\mathbb{P}}=\mathbb{P}_\xi \otimes \tilde \tau$ and $\tilde{\mathbb{E}}$, respectively. Then
\begin{equation}\label{CRE8.5}
\mathbb{E}P_n^{-p/2}W_n^{p/2}=\tilde{\mathbb{E}}W_n^{p/2}(\mathbb{E}m_0^{-p/2})^{n}.
\end{equation}
Combing (\ref{CRE8.5}) with (\ref{CRE2.4.2}), we obtain
$$\sup_n\mathbb{E}|\hat{A}_n|^p\geq C\inf_n\tilde{\mathbb{E}}W_n^{p/2} N^{p/2-1}\sum^{N-1}_{n=0}\left(\rho^{p}\mathbb{E}m_0^{-p/2} \right)^n.$$
Notice that
$$\tilde{\mathbb{E}}\log m_0=\mathbb{E}m_0^{-p/2}\log m_0>0,$$
and
$$\tilde{\mathbb{E}}\frac{Z_1}{m_0}\log^+ Z_1=\mathbb{E}m_0^{-p/2-1}Z_1\log^+Z_1<\infty.$$
Hence $W$ is non-degenerate under $\tilde{\mathbb{P}}$, i.e. $\tilde{\mathbb{P}}(W>0)>0$ (cf. e.g. \cite{a}, \cite{tanny2}), so that $\inf_n\tilde{\mathbb{E}}W_n^{p/2}=\tilde{\mathbb{E}}W^{p/2}>0$. Therefore, $\sup_n\mathbb{E}|\hat{A}_n|^p<\infty$ implies that $\rho(\mathbb{E}m_0^{-p/2})^{1/p}<1$.
\end{proof}

\subsection{Exponential  rate  of  $W_n$}
Again, by the relations of $\hat A_n$, $A(\rho)$ and $\rho^n(W-W_n)$ , combined with  Propositions \ref{CRP2.4.3} and \ref{CRP2.4.1},  we obtain  the following criteria for the  annealed  $L^p$  convergence rate of $W_n$.
\begin{thm}[Exponential  rate of annealed $L^p$ convergence of  $W_n$]\label{CRC5.8}
 Let  $\rho>1$.
\begin{itemize}
\item[] (i) Let $p\in(1,2)$. If $\mathbb{E}\left(\mathbb{E}_\xi\left(\frac{Z_1}{m_0}\right)^r\right)^{p/r}<\infty$ and
$\rho(\mathbb{E}m_0^{p(1/r-1)})^{1/p}<1$ for some $r\in[p,2]$, then
\begin{equation}\label{CRE2.1.2}
(\mathbb{E}|W-W_n|^p)^{1/p}=o(\rho^{-n}).
\end{equation}
Conversely, if (\ref{CRE2.1.2}) holds, then $\rho\leq \exp({\frac{1}{2}\mathbb{E}\log m_0})$; if additionally $\mathbb{E}m_0^{-p/2}\log m_0>0$ and $\mathbb{E}m_0^{-p/2-1}Z_1\log^+Z_1<\infty $, then $\rho(\mathbb{E}m_0^{-p/2})^{1/p}\leq1$.\\

\item[] (ii) Let $p\geq2$. If
$\mathbb{E}(\frac{Z_1}{m_0})^p<\infty\;\text{ and }\;
\rho\max\{(\mathbb{E}m_0^{1-p})^{1/p},(\mathbb{E}m_0^{-p/2})^{1/p}\}<1$, then
(\ref{CRE2.1.2}) holds. Conversely, if (\ref{CRE2.1.2}) holds, then  $
\rho\max\{(\mathbb{E}m_0^{1-p})^{1/p},(\mathbb{E}m_0^{-p/2})^{1/p}\}\leq1$.
\end{itemize}
\end{thm}

Note that
$\rho^{pn}\mathbb{E}|W-W_n|^p\rightarrow0$ implies that
$\forall \rho_1\in(1,\rho)$,
$\rho_1^{pn}\mathbb{E}_\xi|W-W_n|^p\rightarrow0$ a.s. by Borel-Cantelli's lemma and Markov's inequality. So under the conditions of Theorem \ref{CRC5.8},
we can also obtain (\ref{CRE2.1.1}). However,   by
Jensen's inequality,  it can be seen that  the conditions
of Theorem \ref{CRC5.8} are stronger than those of Theorem \ref{CRC5.5}.
 \\*

The proof of Theorem \ref{CRC5.9} is now easy.

\begin{proof}[Proof of Theorem \ref{CRC5.9}]
 Theorem \ref{CRC5.9} is a direct consequence of Theorem \ref{CRC5.8}: taking $r=p$ in Theorem \ref{CRC5.8} gives (a), and taking $r=2$ yields (b).
\end{proof}

\end{document}